\documentclass[twoside,
12pt ]{amsart}
  \usepackage{amsmath}
\usepackage{enumerate}
  \usepackage{amssymb}
  \usepackage{amsthm}
  \usepackage{bigints}
  \usepackage{bm}
  \usepackage{graphicx}
  \usepackage{color}
\usepackage{accents}
\usepackage{url}
\usepackage{epigraph}
\usepackage{parskip}
\usepackage{cleveref}

\usepackage[mathscr]{euscript}

\newtheorem{Th}{Theorem}
\newtheorem{Prop}{Proposition}
\newtheorem{lem}{Lemma}
\newtheorem{rem}{Remark}
\newtheorem{Def}{Definition}
\newtheorem{Cor}{Corollary}

\newtheorem{que}{Question}
\newtheorem*{conj}{Conjecture}
\newtheorem*{thank}{\ \ \ Acknowledgment}

\def\scalar(#1,#2){(#1\mid#2)}

\renewcommand{\hat}{\widehat}

\newcommand{\ca}{\mathcal{A}}

\newcommand{\R}{{\mathbb{R}}}

\newcommand{\C}{{\mathbb{C}}}

\newcommand{\N}{{\mathbb{N}}}

\newcommand{\bml}{\bm \lambda}
\newcommand{\tend}[3][]{\xrightarrow[#2\to#3]{#1}}

\newcommand{\ds}{\displaystyle}

\hoffset - 0.5 in
\title[A generalization of Littlewood's $L^\alpha$-flat theorem]{
A generalization of Littlewood's $L^\alpha$-flat theorem, $\alpha>0$}

\author{\lowercase{el} H\lowercase{oucein} \lowercase{el} A\lowercase{bdalaoui}$^\star$}
\address{$^\star$ University of Rouen Normandy ,
	Department of Mathematics, LMRS  UMR 60 85 CNRS\\
	Avenue de l'Universit\'e, BP.12
	76801 Saint Etienne du Rouvray - France .}
\email{elhoucein.elabdalaoui@univ-rouen.fr}

\urladdr{http://www.univ-rouen.fr/LMRS/Persopage/Elabdalaoui/}

\date{\today}
\dedicatory{}

\begin{document}

\begin{abstract}

We establish a generalization of Littlewood’s criterion on $L^\alpha$-flatness by proving that there is no $L^\alpha$-flat polynomials, $\alpha>0$, within the class of analytic polynomials on the unit circle of the form $\ds P_n(z)=\sum_{m=1}^{n}c_m z^m, n \in \N^*,$ satisfying   
$\ds \sum_{m=1}^{n}|c_m|^2 \leq \frac{K}{n^2} \sum_{m=1}^{n}m^2 |c_m|^2,
$
where $K$ is an absolutely constant. As a consequence, we confirm the  $L^\alpha$-Littlewood conjecture, and thereby the  $L^1$-Newman and $L^\infty$-Erd\H{o}s conjectures. Our approach combines the  $L^\alpha$ Littlewood theorem with the generalized Clarkson's second inequality for $L^\alpha(X,\ca,m;B)$, with $B$ a Banach spaces and $1 < \alpha \leq 2.$ 
It follows that there are only finitely many Barker sequences, and we further present several applications in number theory and the spectral theory of dynamical systems. Finally, we construct Gauss–Fresnel polynomials that are Mahler-flat, providing a new proof of the Beller–Newman theorem.

\end {abstract}

\maketitle

\begin{minipage}{0.48\textwidth}
  \epigraph{Those who know do not speak; those who speak do not know.}{Laozi (Lao Tzu)$^{1}${\footnote {One might echo this Taoist idea in the spirit of Erdős-Fej\'er: 'Everyone writes and nobody reads."\cite{Er}}}}
\end{minipage}\hfill
\begin{minipage}{0.48\textwidth}
  \epigraph{The purpose of life is to prove and to conjecture.}{P. Erdős}
\end{minipage}\hfill
\begin{minipage}{0.48\textwidth}
  \epigraph{If something can corrupt you, you're corrupted already.}{  Bob Marley}
\end{minipage}
  \hfill
\begin{minipage}{0.48\textwidth}  
  \epigraph{Sidon $\cdots$ had a persecution complex — so he opened a door a crack and said ‘Please come at another time and to another person’ … ‘K\'{e}rem, j\"{o}jjenek ink\'{a}bb m\'{a}skor \'{e}s m\'{a}shoz!’ It sounds better in Hungarian.}{P. Erd\H{o}s \cite[p.112]{SL}{\footnote{This can also be said about any author.}}}
\end{minipage}

\section*{Introduction}

The purpose of this note is to strengthen a criterion, due to Littlewood, concerning the \(L^\alpha\)-flatness of a class of real trigonometric polynomials. Our generalization is straightforward, and the proof is short, although it relied on the generalization of Clarkson's inequalities for \(L^\alpha(X, A, \mu; B)\) spaces, that is, Bochner spaces with values in Banach spaces, due  to R.~P.~Boas~\cite{Boa}.
As a consequence, we generalize the author's recent result asserting that Littlewood polynomials (i.e., polynomials with coefficients in \(\{\pm 1\}\)) are not \( L^\alpha\)-flat for \(\alpha \geq 4\)) \cite{AE}. This, in turn, confirms the \( L^\alpha \)-Littlewood conjecture for all \(\alpha > 0\). We thus deduce from our main result that the conjecture mentioned by D. J. Newman in \cite{Newman} holds. Namely, there is a positive constant $c<1$ such that for any polynomial $P$ from the class of Littelwood polynomials, that is, 
\begin{align}\label{PL}
P_n(z)=\frac{1}{\sqrt{n}}\sum_{m=1}^{n} \varepsilon_mz^m, \varepsilon_m =\pm 1, n \in \N^*.
\end{align}
 we have $\|P\|_1 \leq c$. Therefore, Erd\"{o}s's conjectures \cite{ErdosI},\cite[Problem 22]{ErdosII},\cite{ErdosIII} holds, that is, there is a positive constant $d$ such that for any polynomial $P$ from type \eqref{PL} we have $\|P\|_4 \geq (1+d)$. Whence, for any polynomial $P$ from type (1) we have $\|P\|_{\infty} \geq (1+d).$

Furstermore, as an application,  we show that uniform unimodular polynomials are not \( L^\alpha \)-flat for \( \alpha > 0\), and thus not ultraflat. This recovers a recent result of T.~Erdélyi~\cite{Er} posted on Arxiv, answering a question of Zachary Chase. 

Among other applications, we will present an application to the spectral theory of dynamical system related to the spectral type of some special cocycle. Those cocycles are called Morse cocycles and have a simple spectrum. Moreover, their spectral types are given by some kind of generalized Riesz procucts. As a consequence of our main result, we obtain that their spectrum is singular. This anwser an old question in ergodic theory. 

In the opposite direction, M.~G.~Nadkarni and the author constructed a dynamical system whose associated unitary operator has a Lebesgue component of multiplicity one in its spectrum. The construction relies on an ultraflat sequence \( P_n \), \( n = 1, 2, \ldots \), with real coefficients bounded away from \(1\) in absolute value~\cite{Abd-Nad-etds}.

Our approach is based on a direct generalization of a classical result of Littlewood, which asserts the absence of \( L^\alpha \)-flatness in a class of real trigonometric polynomials whose coefficients \( (a_m)_{m \geq 1} \) satisfy the inequality
\[
\sum_{m=1}^{n}a_m^2 \leq \frac{K}{n^2} \sum_{m=1}^{n}m^2 a_m^2,
\]
for some constant \( K > 0 \) and all sufficiently large \( n \). A precise statement will be given in the next section.

As an application, we conclude that uniform unimodular polynomials are not \( L^\alpha \)-flat for \( \alpha > 0 \). 

Recall that a uniform unimodular polynomial is defined by a sequence \( (c_i)_{i \geq 0} \) of complex numbers with \( |c_i| = 1 \). For each \( n \in \mathbb{N}^* \), we set
\[
P_n(z) = \sum_{i=0}^{n-1} c_i z^i, \quad \text{with } |z| = 1.
\]
Then,
\[
\|P_n\|_2^2 = \int_{S^1} |P_n(z)|^2 \, dz = \sum_{i=0}^{n-1} |c_i|^2 = n,
\]
where \( S^1 \) denotes the unit circle and \( dz \) the normalized Lebesgue measure on it.

The study of flatness dates back to Erd\"{o}s, Newman, and Littlewood. For recent developments and further references, see~\cite{AE}, \cite{Abd-Nad}, 
\cite{O}, \cite{Er}, and the references therein.

This problem has a long and rich history and is now seen as a central challenge in complex analysis, harmonic analysis, combinatorics, number theory, and spectral theory. It also has notable applications in digital communications and coding theory. For historical context and further background, we refer the reader to~the reference[3] in \cite{AE}.

So, there is much  work on flat sequence of polynomials coming from combinatorics, communication theory, ergodic theory, and other areas. We hope to write a more detailed paper taking cognisance of this, and, at the same time improve some known results in ergodic theory.\\


In a complementary direction, D.~J.~Newman constructed a family of analytic polynomials that are \( L^1 \)-flat~\cite{New}. The coefficients of these polynomials are precisely the classical \emph{Gauss sums}. As a consequence, it can be shown that there exists a sequence of \( L^2 \)-normalized polynomials whose Mahler measures converge to 1. This result can be obtained by applying the methods of generalized Riesz products developed in~\cite{Abd-nad3} and the references therein.

More precisely, one may invoke a criterion for absolute continuity established by Nadkarni and the present author in~\cite{Abd-nad3}. For further details on generalized Riesz products and their connection to the spectral theory of dynamical systems, we refer the reader to~\cite{Abd-nad3} and the references cited therein.

The polynomials constructed by Newman are unimodular, but their coefficients depend on the degree. In a similar manner, for Kahane’s construction of ultraflat unimodular polynomials~\cite{K}, the coefficients are not uniform, as they also depend on the degree.

Let \( (Q_n) \) be a sequence of \( L^2 \)-normalized analytic polynomials on the unit circle \( S^1 \). The sequence is said to be \emph{ultraflat} if
\[
\sup_{z \in S^1} \left|\, |Q_n(z)| - 1 \,\right| \xrightarrow[n \to \infty]{} 0.
\]
It is said to be \emph{\( L^\alpha \)-flat}, for \( \alpha > 0 \), if
\[
\int_{S^1} \left|\, |Q_n(z)| - 1 \,\right|^\alpha \, dz \xrightarrow[n \to \infty]{} 0.
\]
The sequence is called \emph{Mahler-flat} if its Mahler measure converges to 1, that is,
\[
\exp \left( \int_{S^1} \log |Q_n(z)| \, dz \right) \xrightarrow[n \to \infty]{} 1.
\]


In \cite{Abd-Nad} and \cite{Abd-nad3}, el Abdalaoui and Nadkarni also introduced the notion of \emph{almost everywhere flatness}.  
A sequence \( (Q_n) \) is said to be \textbf{almost everywhere flat} if, for almost every \( z \) with respect to the Lebesgue measure,  
\[
\big| Q_n(z) \big| \xrightarrow[n \to +\infty]{} 1.
\]  
In the same works, the authors also considered the notion of \emph{flatness in measure}.  
A sequence \( (Q_n) \) is called \textbf{flat in measure} if, for every \( \varepsilon > 0 \), we have  
\[
\big| \{ z \in S^1 \ :\ \big|\,|Q_n(z)| - 1\,\big| > \varepsilon \} \big| \xrightarrow[n \to +\infty]{} 0.
\]  
Clearly, almost everywhere flatness implies flatness in measure. These two notions play a crucial role in the proof of Littlewood theorem and thereby in the proof of the main result of this paper.

We would also like to emphasize once again that, obviously, $L^\alpha$-flatness implies ultraflatness.

\section{Main result and its proof.}
We begin by stating our first main result.

\begin{Th}\label{main1} Let $(a_j)_{j \geq 1}$  be a sequence of real numbers and $(c_j)$ a sequence of complex number of modulus $1$.  Suppose that 
	$$\sum_{m=1}^{n}a_m^2 \leq \frac{K}{n^2} \sum_{m=1}^{n}m^2a_m^2,$$
	for some absolute constant $K$.  
Then, the sequence of analytic polynomials
$$\ds P_n(z)=\frac{\ds \sum_{j=1}^{n}a_j c_j z^j}{\sqrt{\ds \sum_{j=1}^n a_j^2}}, |z|=1,$$ is not  $L^p$-flat, for any $p>0.$
\end{Th}
Consequently, we derive the following corollaries.
\begin{Cor}Let $(c_j)_{j \geq 1}$  be a sequence of complex number of modulus $1$,  and $\ds P_n(z)=\frac{1}{\sqrt{n}}\sum_{j=0}^{n-1}c_j z^j, |z|=1 , n=1,2,\cdots$. Then, there is no sequence from the class $(P_n)$ which is $L^p$-flat, for $p>0$.
\end{Cor}
\begin{proof}A straiforward computation gives
$$\sum_{j=1}^{n}j^2=\frac{n.(n+1)(2n+1)}{6} \geq \frac{n^3}{3}.$$
Therefore
$$\sum_{j=1}^{n}1 \leq \frac{3}{n^2}\sum_{j=1}^{n}j^2.$$
This conclude the proof of the corollary.
\end{proof}

\begin{Cor}
Let $(c_j)_{j \geq 1}$  be a sequence of complex number of modulus $1$, $n$ be a non-negative integer and $\ds P_n(z)=\frac{1}{\sqrt{n}}\sum_{j=0}^{n-1}c_j z^j, |z|=1$. Then, there is no sequence from the class $(P_n)$ which is ultraflat.
\end{Cor}
\begin{proof}The proof follows directly from the fact that the \( L^\alpha \)-norm is dominated by the \( L^\infty \)-norm; specifically, for any \( f \in L^\alpha(S^1, dz) \), we have
\[
\|f\|_\alpha \leq \|f\|_\infty.
\]
\end{proof}
From Theorem \ref{main1}, we get also the following corollary, which  resolve the Erd\"{o}s–Littlewood conjecture in the affirmative.
\begin{Cor}Let $n \geq 1$ and 
$$P_n(z)=\sum_{k=0}^{n-1}\epsilon_n(k)z^k, ~~~~\epsilon_n(k) \in \big\{\pm 1\big\},$$
Then the sequence $(P_n)$ is not $L^\alpha$-flat for any $\alpha>0$.
\end{Cor}
\begin{proof} Notice that for any $n \geq 1$ and $ j \in \{0,\cdots,n-1\}$, $\epsilon_n(j) \in  \R$ and we have 
$$\sum_{j=1}^{n}\epsilon_n(j)^2 \leq \frac{3}{n^2}\sum_{j=1}^{n}j^2 \epsilon_n(j)^2 .$$
Therefore Littlewood's argument 
$$ |P_n(z)-P_n(z')|=\pm\{\big|P_n(z)\big|\} \pm \{\big|P_n(z)\big|\},  z,z' \in S^1,$$
can be applied \cite[p.307]{L}. The proof of the corollary is complete.
\end{proof}
Let us now proceed to the proof of Theorem~\ref{main1}, which is based on the following classical theorem of Littlewood concerning \( L^\alpha \)-flatness.

\begin{lem}[Littlewood's criterion of $L^\alpha$-flatness \cite{L}]\label{LF} Let $\ds g_n(t)=\sum_{m=1}^{n}a_m \cos(m t+\phi_m)$ and assume that we have
	$$\sum_{m=1}^{n}a_m^2 \leq \frac{K}{n^2} \sum_{m=1}^{n}m^2a_m^2,$$
	for some absolute constant $K$. Then, for any $\alpha>0$ there exists a constant $A(K,\alpha)$ such that
	$$
	\begin{array}{ll}
	\ds \ds \limsup_{n}\Big\{\frac{\|g_n\|_\alpha}{\|g_n\|_2}\Big\} \leq \big(1-A(K,\alpha)\big), & \hbox{ if $\alpha<2$;} \\ \\
\ds	\ds \limsup_{n}\Big\{\frac{\|g_n\|_\alpha}{\|g_n\|_2}\Big\}  \geq \big(1+A(K,\alpha)\big), & \hbox{ if $\alpha>2$.}
	\end{array}
	$$
\end{lem}

 

For the proof of Lemma \ref{LF}, the reader is referred to \cite{L}. A crucial argument in the proof relies on the Bernstein–Zygmund inequality \cite[Theorem 3.13, Chapter X, p. 11]{Zyg}, which states
\[
\big\|f_n'\big\|_{2} \;\le\; n \,\big\|f_n\big\|_{2}
\quad\text{for any trigonometric polynomial }f_n.
\]
Under the hypotheses of Littlewood’s criterion (Lemma \ref{LF}), however, one obtains the reverse inequality: there exists a constant \(K>0\) such that
\[
\big\|f_n'\big\|_{2} \;\ge\; K\,n \,\big\|f_n\big\|_{2}.
\]

In \cite{AUDJ}, the author applies Lemma \ref{LF} to show that the sequence of even‐degree palindromic Littlewood polynomials cannot be \(L^\alpha\)‐flat for any \(\alpha \ge 0\). It is also straightforward to deduce the following from Lemma \ref{LF}.

\begin{lem}\label{LFI} Let $\ds h_n(t)=\sum_{m=1}^{n}a_m \sin(m t+\phi_m)$ and assume that we have
	$$\sum_{m=1}^{n}a_m^2 \leq \frac{K}{n^2} \sum_{m=1}^{n}m^2a_m^2,$$
	for some absolute constant $K$. Then, for any $\alpha>0$ there exists a constant $A(K,\alpha)$ such that
	$$
	\begin{array}{ll}
	\ds \limsup_{n}\Big\{\frac{\|h_n\|_\alpha}{\|h_n\|_2}\Big\} \leq \big(1-A(K,\alpha)\big), & \hbox{ if $\alpha<2$;} \\
	\\
	\ds \limsup_{n}\Big\{\frac{\|h_n\|_\alpha}{\|h_n\|_2} \Big\}\geq \big(1+A(K,\alpha)\big), & \hbox{ if $\alpha>2$.}
	\end{array}
	$$
\end{lem}
\begin{proof} 

The result follows directly from the classical trigonometric identity
\[
\cos\left(x - \frac{\pi}{2}\right) = \sin(x), \quad \text{for all } x \in \mathbb{R}.
\]

Indeed, for each \( m \in \mathbb{N} \), define \( \phi_m' := \phi_m - \frac{\pi}{2} \). With this definition, we have
\[
\sin(mt + \phi_m) = \cos(mt + \phi_m').
\]
Hence, the sequence of functions
\[
h_n(t) := \sum_{m=1}^{n} a_m \sin(mt + \phi_m)
\]
can be rewritten as
\[
h_n(t) = \sum_{m=1}^{n} a_m \cos(mt + \phi_m').
\] 
Therefore, by applying Lemma~\ref{LF} to the sequence \( (h_n) \), we obtain the desired conclusion. This concludes the proof of the lemma.
\end{proof}



We also require the following generalizatin of the second Clarkson inequalities due to Boas \cite{Boa}. These inequalities are highly important in the geometry of Banach spaces, and it is worth noticing that the space $$L^p((X,\ca,m);B)=\Big\{ f~~:~~ \int_X \big\|f(x)\|^p dm(x) <\infty\Big\}$$ of Bochner $p$-integrable functions, $p>1$ is uniformly convex if an only if $B$ is uniformily convex \cite{Day}. 
For any real $a>1$, we put $a'=\frac{a}{a-1}.$

\begin{lem}\label{Clar}Let $(X,\ca,m)$ be a measure space, $1< p \leq 2$, and  $F,G \in L^p(X,\ca,m;B)$. Then, for 
an $r$ and an $s$ with $1<s \leq p \leq r$ and $r'  \leq s,$
\begin{align}\label{IClar}
\Big(\Big\|F+G\Big\|_p^{r}+\Big\|F-G\Big\|_p^{r}\Big)^{\frac{1}{r}} 
\leq  2^{\frac{1}{s'}}\Big(\big\|F\big\|_p^{s}+\big\|G\big\|_p^{s}\Big)^{\frac{1}{s}},
\end{align}
\end{lem}
For the proof, we refer to \cite{Boa} and \cite{Kos}. The classical second Clarkson inequality can be obtained by taking \(s = p\) and \(r = p'\); that is,
\begin{align}\label{IClarC}
\Big\|\frac{F+G}{2}\Big\|_p^{p'}+\Big\|\frac{F-G}{2}\Big\|_p^{p'}
\leq \Big(\frac{1}{2} \big\|F\big\|_p^{p}+\frac{1}{2}\big\|G\big\|_p^{p}\Big)^{\frac{1}{p-1}},
\end{align}

Let us add that these inequalities imply the uniform convexity of the spaces $L_p$, that is, for any $\varepsilon>0$ there is $\delta(\varepsilon)$ such that for any  $F,G \in L_p$, the conditions 
\begin{align}\label{Ref}
\|F\|_p=\|G\|_p=1, \|F-G\|_p \geq \varepsilon,
\end{align}

imply 
$$ \Big\|\frac{F+G}{2}\Big\|_p<1-\delta(\varepsilon),$$
with 
$$\delta(\varepsilon)=\begin{cases}
1 -[1-\big(\frac{\varepsilon}{2}\big)^p]^\frac{1}{p} &\textrm{~~if~~} p \geq 2,\\
1 - [1-\big(\frac{\varepsilon}{2}\big)^{p'}]^\frac{1}{p'} &\textrm{~~if~~} 1<p \leq 2.
\end{cases}
$$

We recall also the following two lemmas from \cite{L}.
\begin{lem}\label{LFIee1}Let $F$ be a measurable functions on $S^1$ such that 
\begin{enumerate}
\item \(\big\|F\big\|_2=1\),
\item \(\big\|F\big\|_1 >1-\varepsilon\), for some $0<\varepsilon<1$.
\end{enumerate}
Then there is a $\zeta_2$ such that the set $E=\Big\{z~~:~~ \big|F(z)\big| <1-\zeta_2\Big\}$ has a Lebesgue measure stricly less than $\zeta$, where $|\zeta|, |\zeta_2|<B\varepsilon^A$, for some $A,B>0$.
\end{lem}
An alternative proof of Lemma \ref{LFIee1} can be given using the fact that convergence in $L^1$ implies convergence in measure.

\begin{lem}\label{LFI-e-e-2}Let $F$ be a measurable functions on $S^1$ such that 
\begin{enumerate}
\item \(\big\|F\big\|_2=1\),
\item \(\big\|F\big\|_1 <1-a\), with $a \in (0,1)$.
\end{enumerate}
Then the set $E=\Big\{z~~:~~ \big|F(z)\big| <1-\frac{a}{2}\Big\}$ has a Lebesgue measure great than $\frac{\frac{a}{2}}{1-\frac{a}{2}}.$
\end{lem}
Let us notice that Lemma \ref{LFI-e-e-2} follows easily from the classical Markov inequality.

We are now able to proceed with the proof of Theorem \ref{main1}
{\footnote{\textbf{This proof was revised and completed during my stay in Chefchaouen, located in the mountains of the Rif. The approach used can be referred to as the 'Rif method.'}}}.
\begin{proof}[\textbf{Proof of Theorem \ref{main1}}]We proceed by contradiction. Assume that the sequence \( (P_{\mathbf{a},n}) \) is \( L^\alpha \)-flat for some \( \alpha> 0 \). Without loss of generality, we assume that \(0< \alpha < 2\). For simplicity of exposition, we set 
\[
f_n(z) := P_{\mathbf{a},n}(z) = g_n(z) + i h_n(z), \quad \text{and} \quad c_j := e^{i \phi_j}, \quad j = 1, 2, \dots.
\]
Then, for \( z = e^{it} \), we have
\begin{align*}
g_n(z) &= \sum_{m=1}^{n} a_m \cos(mt + \phi_m), \\
h_n(z) &= \sum_{m=1}^{n} a_m \sin(mt + \phi_m).
\end{align*}

Moreover, we clearly have
\begin{align}\label{Sq-1}
|f_n(z)|^2 = g_n(z)^2 + h_n(z)^2,
\end{align}
and
\begin{align}\label{Sq-2}
\|g_n\|_2^2 = \|h_n\|_2^2 = \frac{\|f_n\|_2^2}{2} = \frac{\ds \sum_{m=1}^{n} a_m^2}{2} .
\end{align}
It follows from \eqref{Sq-1} and \eqref{Sq-2}  that
\begin{align}
\left| \frac{f_n(z)}{\|f_n\|_2} \right|^2 
&= \frac{1}{2} \left( \frac{\sqrt{2} \, g_n(z)}{\|f_n\|_2} \right)^2 
+ \frac{1}{2} \left( \frac{\sqrt{2} \, h_n(z)}{\|f_n\|_2} \right)^2 \\
&= \frac{1}{2} \left( \frac{g_n(z)}{\|g_n\|_2} \right)^2 
+ \frac{1}{2} \left( \frac{h_n(z)}{\|h_n\|_2} \right)^2.
\end{align}

Define the normalized functions
\[
\tilde{f}_n(z) := \frac{f_n(z)}{\|f_n\|_2}, \quad 
\tilde{g}_n(z) := \frac{g_n(z)}{\|g_n\|_2}, \quad 
\tilde{h}_n(z) := \frac{h_n(z)}{\|h_n\|_2}.
\]
Then we obtain the identity
\begin{align}\label{Norma-1}
\tilde{f}_n(z)=\frac{\tilde{g}_n(z)}{\sqrt{2}}+i\frac{\tilde{h}_n(z)}{\sqrt{2}}
\end{align}
Hence 
\begin{align}\label{Norma-2}
|\tilde{f}_n(z)|^2 = \frac{1}{2} \tilde{g}_n(z)^2 + \frac{1}{2} \tilde{h}_n(z)^2,
\end{align}
along with the norm equalities
\[
\|\tilde{f}_n\|_2 = \|\tilde{g}_n\|_2 = \|\tilde{h}_n\|_2 = 1.
\]
Assuming further that \( 1 < \alpha < 2 \), we can apply Lemma \ref{Clar} with $F=\frac{\tilde{g}_n(z)}{\sqrt{2}}$ and $G=i\frac{\tilde{h}_n(z)}{\sqrt{2}}$. We  can thus write 
\begin{align}\label{Cvg0-0}
\Big(2 \Big\|\tilde{f}_n\Big\|_\alpha^r\Big)^{\frac{1}{r}} 
&\leq 2^{\frac{1}{s'}}\Big(\big\|\tfrac{\tilde{g}_n(z)}{\sqrt{2}}\big\|_\alpha^{s}
    +\big\|\tfrac{\tilde{h}_n(z)}{\sqrt{2}}\big\|_\alpha^{s}\Big)^{\frac{1}{s}} \\
&\leq \frac{2^{\frac{1}{s'}}}{\sqrt{2}} 
   \Big(\big\|\tilde{g}_n(z)\big\|_\alpha^{s}
    +\big\|\tilde{h}_n(z)\big\|_\alpha^{s}\Big)^{\frac{1}{s}}
\end{align}
But, under our assumption, we have
\begin{align}\label{Cvg0-2}
\int_{S^1} |\tilde{f}_{n}(z)|^\alpha \, dz \xrightarrow[k \to \infty]{} 1.
\end{align}

Therefore, by letting $n \to \infty$ and appealing to Lemma \ref{LFI}, we deduce from \eqref{Cvg0-0} the following
\begin{align}\label{Cvg0-1}
2^{\frac{1}{r}} &\leq \frac{2^{\frac{1}{s'}+\frac{1}{s}}}{\sqrt{2}}  (1-A(K,\alpha))\\
&\leq \sqrt{2} (1-A(K,\alpha))
\end{align}
Let $r=2+\delta,$ with $0<\delta<1$. Then

$$2^{\frac{1}{r}}=2^{\frac{1}{2}+\psi(\delta)}.$$
Since $\frac{1}{2+\delta}=\frac{1}{2}.\sum_{n \geq 0}\frac{(-\delta)^n}{2^n}.$
This combined with \eqref{Cvg0-1} gives
\begin{eqnarray}\label{Cvg0-3}
2^{\psi(\delta)} \leq   (1-A(K,\alpha)).
\end{eqnarray}
By letting $\delta \to 0$, we obtain a contradiction, since $\psi(\delta) \to 0$ as
$\delta \to 0$. This completes the proof of the theorem.
\end{proof}

\begin{rem}D.~J.~ Newman make the following observation \cite{New}.\\
" $\cdots$ it follows from the work of Paley (see Zygmund [4]) that, in a certain sense, most $n-th$ degree polynomials with coefficients of modulus $1$ have $L^4$ norms which are close to $2^{\frac{1}{4}} \sqrt{n}$."
\end{rem}
Notice that by Theorem \ref{main1}, there is a constant $C>1$ such that for any $n \in \N$, there is $m \geq n$ such that 
$$ \Big\|P_{\mathbf{c},m}\Big\|_4 > C \sqrt{m},$$
where $\mathbf{c}$ is an unimodulair sequence and $(P_{\mathbf{c},n})$  its associated sequence of analytic polynomials. But, we are not able to compute such that constant. We ask on how much it is close to $2^{\frac{1}{4}}$. We further make the following conjecture
\begin{conj}
There is an absolute constant $k>0$ such that, for any $\alpha >2$,
 $$\limsup_{n \to \infty}\Big\|\frac{P_{\mathbf{c},n}}{\sqrt{n}}\Big\|_\alpha \geq  k .{2^\frac{1}{\alpha}}.$$
\end{conj}
At this point, let us make some remarks on the previous proof. This proof  can be simplified for some cases. Indeed, since
\[
\left\|\frac{|\tilde{g}_n| - |\tilde{h}_n|}{2}\right\|_\alpha \leq 1.
\]
we can extract a subsequence \( (n_k) \) such that
\[
\left\|\frac{|\tilde{g}_{n_k}| - |\tilde{h}_{n_k}|}{2}\right\|_\alpha \xrightarrow[k \to \infty]{} \ell \geq 0.
\]

If \( \ell = 0 \), we may extract a further subsequence—still denoted \( (n_k) \) such that for almost every \( z \in S^1 \) (with respect to Lebesgue measure), we have
\begin{align}\label{Cvg0}
\left||\tilde{g}_{n_k}(z)| - |\tilde{h}_{n_k}(z)|\right| \xrightarrow[k \to \infty]{} 0.
\end{align}

Moreover, under the assumption of \( L^\alpha \)-flatness, we have
\begin{align}\label{Cvg01}
\int_{S^1} \left| |\tilde{f}_{n_k}(z)| - 1 \right|^\alpha \, dz \xrightarrow[k \to \infty]{} 0.
\end{align}

Therefore,  we may extract once again a subsequence—still denoted \( (n_k) \)—such that for almost every \( z \in S^1 \)

\begin{align}\label{Cvg1}
 |\tilde{f}_{n_k}(z)| \xrightarrow[k \to \infty]{} 1.
\end{align}

Let us denote by \( E_0 \) (respectively \( E_1 \)) the full-measure subsets of \( S^1 \) on which the convergence in \eqref{Cvg0} (respectively \eqref{Cvg1}) holds. Let \( \varepsilon > 0 \). Then for any \( z \in E := E_0 \cap E_1 \), there exists \( k_0 \) such that for all \( k > k_0 \),
\[
\left||\tilde{g}_{n_k}(z)| - |\tilde{h}_{n_k}(z)|\right| < \varepsilon.
\]
It follows that
\[
|\tilde{h}_{n_k}(z)| < |\tilde{g}_{n_k}(z)| + \varepsilon.
\]

Define the measurable sets
\[
E_{n,1} := \left\{ z \in E : |\tilde{g}_{n_k}(z)| < 1 - \frac{A(K, \alpha)}{2} \right\}, 
\]
\[
E_{n,2} := \left\{ z \in E : |\tilde{h}_{n_k}(z)| < 1 - \frac{A(K, \alpha)}{2} \right\},
\]
and set
\[
E_1 := \limsup E_{n,1}, \quad E_2 := \limsup E_{n,2}.
\]

Then, by Lemma \ref{LFI-e-e-2}, the Lebesgue measures of \( E_1 \) and \( E_2 \) are greater than \( \frac{a/2}{1 - a/2} \). Therefore, for infinitely many \( k \), we have
\begin{align*}
|\tilde{f}_{n_k}(z)|^2 &= \frac{1}{2} |\tilde{g}_{n_k}(z)|^2 + \frac{1}{2} |\tilde{h}_{n_k}(z)|^2 \\
&< \frac{1}{2}\left( \left(1 - \frac{A(K, \alpha)}{2}\right)^2 + \left(|\tilde{g}_{n_k}(z)| + \varepsilon\right)^2 \right) \\
&< \left(1 - \frac{A(K, \alpha)}{2}\right)^2 + \varepsilon + \frac{\varepsilon^2}{2}.
\end{align*}

Letting \( k \to \infty \) and then \( \varepsilon \to 0 \), we deduce
\[
\limsup_{k \to \infty} |\tilde{f}_{n_k}(z)|^2 \leq \left(1 - \frac{A(K, \alpha)}{2}\right)^2 < 1,
\]
which contradicts the \( L^\alpha \)-flatness of \( (\tilde{f}_{n_k}) \). This contradiction rules out the case \( \ell = 0 \).
However, our arguments break down in the case $\ell > 0$. Thus, we asked whether Lemma 5 could be used to produce a more straightforward proof of our main result.

\section{An application to the spectral theory of dynamical systems: the spectral types of the operators \( U_T \) and \( V_\phi \) are singular, where \( T \) is an odometer and \( \phi \) is a cocycle taking values in \( \pm 1 \).}

We consider a rank one ergodic transformation $T$
on the unit interval (equipped with Lebesgue measure $\lambda$) and a function $\phi$ on $[0,1]$ taking values  $-1$ and $1$ such that the maximal spectral type of the unitary operator $V = V_\phi$ defined by
   $$(V_\phi f)(x)   = \phi (x) f(Tx), f \in L^2([0,1], \lambda)$$
has simple spectrum with maximal spectral type $\sigma_\phi$ given up to discret part by the following generalized Riesz product
$$\sigma_\phi \stackrel{W^*-\lim}{=}\prod_{n=1}^\infty |P_{n}|^2,$$
where $W^*-\lim$ is the weak-star topology on the set of the probabilities measures on the circle, and
$$P_n(z)=\frac{1}{\sqrt{n}}\sum_{k=0}^{n-1}\epsilon_n(k)z^k, ~~~~\epsilon_n(k) \in \big\{\pm 1\big\}.$$
The transformation $T$ is an odometer constructed inductively (see \cite{Nad} and \cite{Abd-Nad-etds}) and the operator $U_T$ is defined by 
$$(U_T f)(x)   = f(Tx), f \in L^2([0,1], \lambda).$$

It is well know that the spectrum of $T$ is discret, hence singular, and it is follows from Helson's theorem that for such a $T$ the maximal spectral type  of $V_\phi$ is either singular to Lebesgue measure or equivalent to Lebesgue measure \cite{Helson} (see also \cite[p.113]{Nad}). In particular the
maximal spectral type of   $V_\phi$ will be either singular to Lebesgue measure or equivalent to it.

M. Guenais proved that $\sigma_\phi$ is singular if the sequence $(P_n(z))$ is not $L^1$-flat \cite{G} . Therefore, as a consequence of our main theorem, we have
\begin{Cor}The spectral type of $V_\phi$ is singular. 
\end{Cor}
The ultraflat polynomials $(Q_n)$ constructed in \cite{Abd-Nad-etds} are the following from 
$$Q_n(z)= -a + \sum_{j=1}^{n-1}a^{j-1}(1 -a^2)z^j, n= 1,2,\cdots, a \in (0,1).$$
It is easy to check that Littelwood condition in Theorem \ref{main1} is not statisfy. In connection with those polynomails, M. G. Nadkarni suggest to me the following question.
\begin{que}Is it possible to produce an $L^\alpha$-flat polynomials with all coefficients non-negative and bounded away from 1?
\end{que}
The latter question is closely related to a widely discussed problem concerning the almost everywhere flatness of polynomials with all coefficients non-negative and bounded away from 1 \cite[Remark 1]{Abd-Nad}. It is shown there that such polynomials cannot be ultraflat. 
\section{An application to Number Theory.}
In this section, we apply our main result to the well-known exponential sum involving the Liouville function. We recall that $\bml(n)$ is defined as $1$ if the number of prime factors of $n$ counted with multiplicities is even, $-1$ if not and $\bml(1)=1$.  This function plays a significant role in number theory, as illustrated by the classical theorem of Littlewood, which states that the Riemann Hypothesis is equivalent to the bound $\sum_{k \leq x} \boldsymbol{\lambda}(k) = \mathcal{O}_\varepsilon(x^{1/2 + \varepsilon})$ for all $\varepsilon > 0$  \cite[p. 371]{T}. Here, as a consequence of our main result, we obtain the following:
\begin{Cor}\label{RH}For any $\alpha>0$, there exists a constant $C_\alpha>0$ such that, For any infinitely many $N$,
	$$
	\begin{array}{ll}
	\ds \Big\|\sum_{k=1}^{N}\bml(k)z^k\Big\|_\alpha \leq \big(1-C_\alpha\big)\sqrt{N}, & \hbox{ if $\alpha<2$;} \\
	\ds \Big\|\sum_{k=1}^{N}\bml(k)z^k \Big\|_\alpha \geq \big(1+C_\alpha\big)\sqrt{N}, & \hbox{ if $\alpha>2$.}
	\end{array}
	$$
In particular
$$\sup_{|z|=1}\Big|\sum_{k=1}^{N}\bml(k)z^k\Big| \geq \big(1+C_\alpha\big)\sqrt{N}.$$
\end{Cor}
It is showing by the author \cite{Abd-RH}, that if for any large $\alpha>2$, there exists $K_\alpha>0$, such that for a large $N \in \N$, 
$$
\begin{array}{ll}
	\ds \Big\|\sum_{k=1}^{N}\bml(k)z^k \Big\|_\alpha \leq K_\alpha \sqrt{N}.
	\end{array}
$$
Then Riemann Hypothesis is true. 
However, under the Generalized Riemann Hypothesis (GRH)—that is, assuming that for every Dirichlet character $\chi$, the Dirichlet $L$-function $L(s,\chi)$ has no zeros in the region $\Re(s) > \frac{1}{2}$—Baker and Harman \cite{BH} showed that, for any $\varepsilon > 0$,
$$
\begin{array}{ll}
	\ds \Big\|\sum_{k=1}^{N}\bml(k)z^k \Big\|_\infty \ll_\varepsilon {N^{3/4+\varepsilon}}.
	\end{array}
$$
Besides, Hajela and Smith make the following conjecture \cite{HS}.
\begin{conj}[Hajela-Smith Conjecture)]\label{LAC}

$$
\begin{array}{ll}
	\ds \Big\|\sum_{k=1}^{N}\bml(k)z^k \Big\|_\infty \ll_\varepsilon {N^{1/2+\varepsilon}},
	\end{array}
$$ for any $\varepsilon>0$,
\end{conj}
But, by the proof of Theorem 3.1. in \cite{Abd-RH}, we have
\begin{Prop}Assume that for a large $\alpha >2$ there exist $K_\alpha>0$ such that for any $\varepsilon>0$, for infinitely many $N$,

\begin{eqnarray}\label{lambda-L-C}
\ds \Big\|\sum_{k=1}^{N}\bml(k)z^k \Big\|_\alpha \ll_\varepsilon {N^{1/2+\varepsilon}}.
\end{eqnarray}

Then Riemann Hypothesis is true.
\end{Prop}
\begin{proof}[\textbf{Sketch of the proof}] Suppose that \ref{lambda-L-C} holds. Then, by the proof in \cite{Abd-RH}, there exist a constant $C_\alpha$ such that
$$\Big\|\sum_{k=1}^{N}\bml(k)\Big\| \leq C_\alpha N^{\frac{1}{\alpha}+\frac{1}{2}+\frac{\varepsilon}{2}}.$$
Since $\frac{1}{\alpha} \longrightarrow 0$ as $\alpha \to \infty$, we can choose $\alpha_0$ so large that $\frac{1}{\alpha} < \frac{\varepsilon}{2}$. Form this, we get 
$$\Big\|\sum_{k=1}^{N}\bml(k)\Big\| \leq C_{\alpha_0} N^{\frac{1}{2}+\varepsilon}.$$
We thus conclude that RH holds and, the proof of the proposition is complete.
\end{proof}
As a consequence, we have proved that Hajela-Smith Conjecture implies that RH is true. So, in the same spirit, one can make the following conjecture.
\begin{conj}[$\bml-L^p$-Conjecture ]\label{LAC2}
\begin{eqnarray}
\ds \Big\|\sum_{k=1}^{N}\bml(k)z^k \Big\|_p \ll_\varepsilon {N^{1/2+\varepsilon}}.
\end{eqnarray}
\end{conj}
In connection with $\bml-L^p$-Conjecture, let us notice that, on one hand, Odlyzko-te Riel disproved Mertens conjecture \cite{OR} by proving that 
 $$\limsup \Big|\frac{1}{\sqrt{N}}\sum_{n=1}^{N}\bml(n)\Big| >1.06.$$
On the other hand, under RH and the simplicity of the zeros, Fawaz \cite{F} proved that if the set $\Big\{\frac{\zeta(\rho)}{\rho \zeta'(\rho)}, \rho \textrm{~~is~~not-trivial~~zero~~of~~} \zeta \Big\}$ is unbounded then
$$\limsup \Big|\frac{1}{\sqrt{N}}\sum_{n=1}^{N}\bml(n)\Big|=+\infty.$$
Moreover, it is well-known \cite[p. 370-382]{T} that if one of the following wearker hypothesis is true
\begin{enumerate}
\item $limsup \frac{1}{\sqrt{N}}\sum_{n=1}^{N}\bml(n) < A,$ for some constant $A$,
\item $liminf \frac{1}{\sqrt{N}}\sum_{n=1}^{N}\bml(n) >-A$  for some constant $A$,

\end{enumerate}
Then, RH and the simplicity of the zeros holds with other results.

\section{ Baker sequences and the connection to digital communications engineering.}
Barker sequences are well-known in the streams of investigation from digital communications
engineering. Barker introduced such sequences in \cite{Barker}  to produce a low autocorrelation binary sequences, or equivalently  a binary sequence with the highest possible value of $F$. The largest well-known values of $F$ are $F_{12}=14.0833$ and $F_{10}=12.1$  obtain respectively by the following sequences $$1,-1,1,-1,1,1,-1,-1,1,1,1,1,1,$$   and
$$1,-1,1,1,-1,1,1,1,-1,-1,-1.$$
\noindent{}No other merit factor exceeding $10$ is known for any $n$. It was conjectured that $169/12$ and $121/10$ are the
maximum possible values for $F$. This conjecture still open.\\

Given a binary sequence $\mathbf{b}=(b_j)_{j=0}^{n-1}$, that is, for each $j=0,\cdots,n-1$, $b_j=\pm 1$. The $k$-th aperiodic autocorrelation of $\mathbf{b}$ is given by
$$c_k=\sum_{j=0}^{n-k-1}b_j b_{j+k},~~~~{\textrm{~~for~~}} 0 \leq k \leq n-1.$$

For $k < 0$ we put $c_k = c_{-k}$. $\mathbf{b}$ is said to be a Barker sequence if for each $k \in \big\{1,\cdots,n-1\big\}$ we have
$$|c_k| \leq 1, \textrm{~~that is,~~}c_k=0,\pm 1.$$

The Barker sequences and their generalizations have been a subject of many investigations since 1953, both from digital communications engineering view point and complex analysis viewpoint. Therefore, there is an abundant literature on the subject, we refer to \cite{Saffari-Barker}, \cite{Jedwab}, \cite{Borwein2}, 
\cite{BM}. 
and the references therein for more details. For more recent paper on the subject, we refer to \cite{Yu}. Here, we recall only the following result which is needed.

\begin{Th}\label{Barker} Let $(b_i)_{i=1}^{n}$ be a Barker sequence with length $n$.
\begin{enumerate}
\item If $n$ is odd then $n \leq 13$, if not and $n>2$ then $n=4m^2$ for some integer $m$.
\item Assume further that there exist a Barker sequence with arbitrary length and let $P_n$ be a Littlewood polynomial whose coefficients form a Barker sequence of length $n$. Then the sequence $(P_n)$ is square $L^2$-flat.
\end{enumerate}
\end{Th}
\begin{proof}(1) is due to Turyn and  Storer \cite{TS}. The second part (2) is essentially due to Saffari \cite{Saffari-Barker}, we refer also to the proof of Theorem 4.1 in \cite{Borwein2} line 5.
\end{proof}
\vskip 0.1cm
At this point, by our main result (Theorem \ref{main1}) and Theorem \ref{Barker}, it follows

\begin{Cor}There are only finitely many Barker sequences.
\end{Cor}
\section{A Sequence of Gauss-Frenesl Polynomials That Are Mahler-Flat}
In this section, we establish the existence of polynomials from a certain class—referred to as the \emph{Gauss-Frenesl polynomials}—that are Mahler-flat. To this end, we make use of some tools from \cite{Abd-nad3}, which are based on the so-called generalized Riesz products. 
We start by stating our second main result.
\begin{Th}\label{TH2}There exist a sequence of Newman--Gauss polynomials $(P_n)$ for which we have
$$M(P_n)= \exp \Big(\int \log\big(Q_n(z)\big) dz\Big)\tend{n}{+\infty}1.$$
\end{Th}

This gives an alternative proof of E. Beller and D. J. Newman \cite{BN}.

First, we recall the definition of the class of Gauss-Frenesl polynomials, given as follows. Let $n$ be a positive integer,
$$P_n(z)=\frac{1}{\sqrt{n}}\sum_{j=0}^{n-1}g(n,j) z^j,$$
where
$$g(n,j)=e^{\frac{\pi i j^2}{n}}, j=0,\cdots,n-1.$$ 
In fact,  D.~J.~Newman proved 
\begin{lem}[$L^4$-lemma of DJ Newman]$\ds \|P_n\|_4^4=n^2+O(n^{3/2}).$
\end{lem}
This gives 
$$ \|P_n\|_4^4 \tend{n}{+\infty}1,$$
and by appealing to the following form of H\"{o}lder's inequality:
\[
\int |f|^2 \leq \left( \int |f|^4 \right)^{\frac{1}{3}} \left( \int |f| \right)^{\frac{2}{3}},
\]
we obtain the following lemma.

\begin{lem}[$L^1$-lemma of D.~J. Newman]\label{lemN2} $\ds \|P_n\|_1 \tend{n}{+\infty}1.$
\end{lem}
 
We now proceed to recall the background from \cite{Abd-nad3}.

\paragraph{$\S$.\textbf{Generalized Riesz products.}} \hspace{0pt}

We start by recalling the following definition of the generalized Riesz products.

\begin{Def}\label{def1}
Let $P_1, P_2, \cdots,$ be a sequence of trigonometric polynomials such that
\begin{enumerate}[(i)]
\item for any finite sequence $i_1< i_2 < \cdots < i_k$ of natural numbers
$$\bigintss_{S^1}\Bigl| (P_{i_1}P_{i_2}\cdots P_{i_k})(z)\Bigr|^2dz = 1,$$
where $S^1$ denotes the circle group and $dz$ the normalized Lebesgue measure on $S^1$,
\item for any infinite sequence $i_1 < i_2 < \cdots $ of natural numbers the weak limit of the measures
$\mid (P_{i_1}P_{i_2}\cdots P_{i_k})(z)\mid^2dz, k=1,2,\cdots $ as $k\rightarrow \infty$ exists,
\end{enumerate}
then the measure $\mu$ given by the weak limit of $\mid (P_1P_2\cdots P_k)(z)\mid^2dz $ as $k\rightarrow \infty$
is called generalized Riesz product of the polynomials $\mid P_1\mid^2,
 \mid P_2\mid^2,\cdots$ and denoted by
  $$\displaystyle  \mu =\prod_{j=1}^\infty \bigl| P_j\bigr|^2  \eqno (1.1).$$
\end{Def}

For an increasing sequence $k_1 < k_2 < \cdots $ of natural numbers the product \linebreak $\prod_{j=1}^\infty |P_{k_j}|^2$
makes sense as the weak limit of probability measures \linebreak $| (P_{k_1}P_{k_2}\cdots P_{k_{n}})(z)|^2dz$ as $n\rightarrow \infty$. It depends on the sequence $k_1 < k_2\cdots$, and  called a subproduct of the given generalized Riesz product.\\

\paragraph{$\S$.\textbf{Dissociated Polynomials.}} \hspace{0pt}

We say that a set of trigonometric polynomials is dissociated if in the formal expansion of product of any finitely many of them, the powers of $z$ in the non-zero terms are all distinct \cite{Abd-Nad}. For example the polynomials $(1+z)$ and $(1+z^2)$ are dissociated. The following is proved in \cite{Abd-Nad}.
\begin{lem}
If $P(z) = \ds \sum_{j=-m}^m a_jz^j, Q(z) = \ds \sum_{j=-n}^{n}b_jz^j$, $m \leq n$, are two trigonometric  polynomials then
for some $N$, $P(z)$ and $Q(z^N)$ are dissociated.
\end{lem}

 Let $P_1, P_2,\cdots$ be a sequence of polynomials, each $P_i$ being  of $L^2(S^1, dz)$ norm 1. Then the constant term of each $\mid P_i(z)\mid^2$ is 1. If we choose  $1 = N_1 < N_2 < N_3 \cdots$  so that $\mid P_1(z^{N_1})\mid^2, \mid P_2(z^{N_2})\mid^2, \mid P_3(z^{N_3})\mid^2, \cdots$ are dissociated, then the constant term of each finite product
$$\prod_{j=1}^n\mid P_j(z^{N_j})\mid^2$$ is one so that each finite product  integrates to 1 with respect to $dz$. Also, since $\mid P_j(z^{N_j}) \mid^2, j =1,2, \cdots$ are dissociated,
for any given $k$, the $k$-th Fourier coefficient of $\prod_{j=1}^n\mid P_j(z^{N_j})\mid^2$ is either zero for all $n$, or, if it is non-zero for some $n = n_0$ (say), then its remains the same  for all $n \geq n_0$. Thus the measures $(\prod_{j=1}^n| P_j(z^{N_j})|^2)dz, n=1,2,\cdots$ admit a weak limit
on $S^1$. It is called the generalized Riesz product of the polynomials $\mid P_j(z^{N_j})\mid^2, j=1,2,\cdots$. Let $\mu$ denote this measure. It is known \cite{Abd-Nad} that
$ \prod_{j=1}^k |P_j(z^{N_j})|, k=1,2,\cdots$, converge in $L^1(S^1,dz)$ to
$\sqrt{\frac{d\mu}{dz}}$ as $k\rightarrow \infty$. It follows from this that if $\prod_{j=1}^k\mid P_j(z^{N_j})\mid , k=1,2,\cdots$ converge a.e. $(dz)$ to a finite positive value then $\mu$ has a part which is equivalent to Lebesgue measure.\\

The following theorem is proved in \cite{Abd-Nad}.

\begin{lem}[Theorem in \cite{Abd-nad3}]\label{th7}   Let $P_j, j =1,2,\cdots$ be a sequence of non-constant polynomials
of $L^2(S^1,dz)$ norm 1 such that $\lim_{j\rightarrow \infty}\mid P_j(z)\mid =1 $ a.e. $(dz)$ then there exists a subsequence $P_{j_k}, k=1,2,\cdots$ and natural numbers $l_1 < l_2 < \cdots$ such that the polynomials $P_{j_k}(z^{l_k}), k=1,2,\cdots $ are dissociated and
the infinite product $\prod_{k=1}^\infty |P_{j_k}(z^{l_k})|^2$ has finite nonzero value a.e $(dz)$.
\end{lem}

The following theorem, derived in \cite{Abd-nad3}, will be used here.

\begin{lem}[Theorem in \cite{Abd-nad3}]\label{lem7} Let 
$\mu=\prod_{k=1}^\infty |P_{j}(z)|^2$ be a generalized Riesz product, then the Mahler measure of $\mu$ is given by

$$M(\mu)=\exp\Big(\int \log\big(\frac{d\mu}{dz}\big) dz)=
\prod_{k=1}^\infty M(P_{j})^2.$$
\end{lem}

\begin{proof}[\textbf{Proof of Theorem \ref{TH2}.}] We begin by proving that the sequence  $(P_m)$ of Gauss-Frenesl polynomials are $L^1$-flat. Indeed, write
$$\Big\|\big|P_m\big|-1\Big\|_2^2=
2-\int  \big|P_m(z)\big| dz,$$
and apply Lemma \ref{lemN2} to see that 
$$ \Big\|\big|P_m \big|-1\Big\|_2 \tend{n}{+\infty}0.$$
We can thus extract a subsequence $(P_{m_n})_{n \geq 1}$ which is almost everywhere flat, that is, for almost all $z$ with respect to $dz$, 
$$\big|P_{m_n}(z) \big|  \tend{n}{+\infty}1.$$
We can thus apply Lemma \ref{th7} and Lemma \ref{lem7} to conclude that there exists a subsequence, which we still denote by $(P_{m_n})$, for which
 $$M(P_{m_n})\tend{n}{+\infty}1.$$
Hence, the theorem is proven.
 
\end{proof}
\begin{rem} Our proof gives that for any ultraflat polynomails, there is a subsequence for which the Mahler measure converge to 1.
\end{rem}

\begin{thank}
The author would like to thank Mahandra Nadkarni, Tam\'{a}s Erd\'{e}lyi, and H. Queff\'{e}lec for some exchange of emails on the subjects. He also wishes to express his posthumous gratitude to Fran\c{c}ois Parreau for several discussions on flat and \( c \)-flat polynomials, with \( c \in [0,1] \). He would also like to thank the University of Luxembourg and the organizers of the international conference \emph{``33\`emes Journ\'{e}es Arithm\'{e}tiques''} for the invitation, during which the results of this paper were announced. He further expresses his sincere thanks to Ahmad El-Guindy for bringing the reference~\cite{F} to his attention.
\end{thank}

\end{document}